\documentclass{aptpub}

\authornames{A. Bhaskar, J. Kamm, Y. S. Song} 
\shorttitle{Approximate sampling formulas for general mutation models} 

\usepackage{times}
\usepackage{amsmath, amssymb, euscript, mathrsfs}
\usepackage{graphicx, cite,subfigure}
\usepackage{caption}

\allowdisplaybreaks

\renewcommand{\crnotice}{}    

\usepackage{mathtools}
\mathtoolsset{showonlyrefs}
\usepackage{fixltx2e}
\MakeRobust{\eqref}

\def\math#1{\mathchoice
        {\mbox{\boldmath$#1$}}%
        {\mbox{\boldmath$#1$}}%
        {\mbox{\boldmath$\scriptstyle#1$}}%
        {\mbox{\boldmath$\scriptscriptstyle#1$}}}%

\def\e{\math{e}}
\def\m{\math{m}}
\def\n{\math{n}}

\def\bfP{\math{P}}

\def\bfPhat{\math{\widehat{P}}}

\def\bfzero{\math{0}}

\def\bfpihat{\math{\widehat{\pi}}}
\def\fall#1#2{(#1)_{#2\downarrow}}
\def\rise#1#2{(#1)_{#2\uparrow}}

 \def\qt{Q}
 \def\rt{R}

\newcommand{\bo}[1]{\ensuremath \boldsymbol{#1}}
\def\m{\bo{m}}
\def\n{\bo{n}}
\def\e{\bo{e}}
\def\t{\bo{t}}
\def\O{{|\mathcal{O}_{\n}|}}
\def\Oset{{\mathcal{O}_{\n}}}
\def\Om{{|\mathcal{O}_{\m}|}}
\def\C{\mathcal{C}}
\def\H{\mathcal{H}}
\def\F{\mathcal{F}}

\def\I{\mathbb{I}}

\def\bbZ{\mathbb{Z}}

\def\bbE{\mathbb{E}}

\def\bbP{\mathbb{P}}

\def\Pn{\mathbb{P}_{\n}}
\def\En{\mathbb{E}_{\n}}

\DeclareMathOperator*{\err}{\text{Err}}
\def\qapprox{q_{\text{approx}}}
\DeclareMathOperator*{\exerr}{\text{AvgErr}}
\DeclareMathOperator*{\worsterr}{\text{WorstErr}}

\newcommand{\sref}[1]{Section~\ref{#1}}

\newcommand{\factref}[1]{Fact~\ref{#1}}
\newcommand{\cref}[1]{Chapter~\ref{#1}}

\newcommand{\propref}[1]{Proposition~\ref{#1}}
\newcommand{\thmref}[1]{Theorem~\ref{#1}}
\newcommand{\corref}[1]{Corollary~\ref{#1}}

\newtheorem{fact}{Fact}
\def\qed{\ensuremath{\Box}}

\def\revisedv2#1{{\bf #1}}

\newtheorem{corrol}[theorem]{Corollary}
\newtheorem{propos}[theorem]{Proposition}

\newcommand{\emailthree}[1]{\footnote{\hspace*{-14pt}$^{\dagger}\,$Email address: #1}\par}
\newcommand{\authoronetwoemail}[2][]{\hspace*{9pt}{\small\textrm{\uppercase{#2}},$^{*,**,\dagger}$ \textit{#1}}\par}
\begin{document}
\title{Approximate sampling formulae for\\ general finite-alleles models of mutation}
 
\authorone[University of California, Berkeley]{Anand Bhaskar}
\authortwo[University of California, Berkeley]{John A. Kamm}
\authoronetwoemail[University of California, Berkeley]{Yun S. Song}

\addressone{Computer Science Division, University of California, Berkeley, CA 94720, USA.}
\addresstwo{Department of Statistics,  University of California, Berkeley, CA 94720, USA.}
\emailthree{yss@stat.berkeley.edu}

\begin{abstract}	
Many applications in genetic analyses utilize sampling distributions, which describe the probability of observing a sample of DNA sequences randomly drawn from a population.  In the one-locus case with special models of mutation such as the infinite-alleles model or the finite-alleles parent-independent mutation  model, closed-form  sampling distributions under the coalescent have been known for many decades.  However, no exact formula is currently known for more general models of mutation that are of biological interest.  
In this paper, models with finitely-many alleles are considered, and an urn construction related to the coalescent is used to derive approximate closed-form sampling formulas for an arbitrary irreducible recurrent mutation model or for a reversible recurrent mutation model, depending on whether the number of distinct observed allele types is at most three or four, respectively.
It is demonstrated empirically that the formulas derived here are highly accurate when the per-base mutation rate is low, which holds for many biological organisms.
\end{abstract}


\section{Introduction}

An important problem in genetic analyses concerns computing the probability of observing a randomly drawn sample of chromosomes under a given model of evolution.  Popular applications of this probability computation include maximum likelihood estimation of model parameters and ancestral inference (see \cite{ste:2001} for a nice introduction).  The coalescent \cite{kin:1982:SPA,kin:1982:JAP} is a useful mathematical framework for performing model-based full-likelihood analyses, but in most cases it is intractable to obtain a closed-form formula for the probability of a given dataset.  A well-known exception to this complication is the celebrated Ewens sampling formula (ESF)~\cite{ewe:1972}, which describes the stationary probability distribution of a sample configuration under the one-locus \emph{infinite}-alleles model in the coalescent or the diffusion limit.  A P\'olya-like urn model interpretation~\cite{hop:1984} of the formula has been known for some time, and recently a new combinatorial proof of the ESF has been provided \cite{gri:les:2005}. Furthermore, the ESF also arises in several interesting contexts outside biology, including random partition structures; the ESF is a special case of the two-parameter sampling formula \cite{pitman:1992,pitman:1995} for exchangeable random partitions.  See \cite{ABT:book} for examples of other interesting combinatorial connections.

In the case of \emph{finitely}-many alleles, a closed-form sampling formula is known \cite{wri:1949} only for the parent-independent mutation (PIM) model, in which the probability of mutating from allele $j$ to allele $i$  depends only on the child allele $i$.  For a general non-PIM mutation model, finding an exact, closed-form sampling formula has remained a challenging open problem.

In this paper, we make progress on this problem by deriving approximate, closed-form sampling formulas that are highly accurate when the mutation rate is low.  
More precisely, given a sample configuration $\n$ and the model parameters (mutation rate $\theta$ and transition matrix $\bfP$), we consider the Taylor expansion of the sampling probability $q(\n\mid\theta,\bfP)$ about $\theta=0$.  As discussed later, if $\bfP$ is irreducible when restricted to the observed alleles in the sample, then the leading order term in the expansion is proportional to $\theta^{\O-1}$, where $\O$ is the number of distinct observed alleles in the sample configuration $\n$.  Hence, 
\begin{align} \label{eq:q_expansion}
q(\n \mid \theta, \bfP) = \theta^{\O-1} \qt(\n \mid \bfP) + O(\theta^{\O}),
\end{align}
where $\qt(\n \mid \bfP)$ is the leading order coefficient that depends on the mutation transition matrix $\bfP$ but not on the mutation rate $\theta$.
In this paper, we consider the problem of obtaining exact closed-form formulas for $\qt(\n \mid \bfP)$.
As many organisms typically have small per-base mutation rates, our results are of biological interest.

By restricting the set of events in the coalescent genealogy for a given sample, Jenkins and Song \cite{jen:son:2011:tpb} provided closed-form formulas for $\qt(\n \mid \bfP)$ for an arbitrary transition matrix $\bfP$ when $\O \leq 3$. In this paper, we provide new proofs of those results, and extend them by supplying a closed-form formula for $\qt(\n \mid \bfP)$ when $\O = 4$ and the transition matrix $\bfP$ is reversible restricted to the observed alleles. 
We prove our results using martingale arguments and use an urn construction related to the coalescent to develop a recursion for the approximate sampling probability, which can then be solved in closed-form using combinatorial techniques.
As a corollary of our results, it can be seen that the simple general formula in \cite[Theorem 6.3]{jen:son:2011:tpb} for $\qt(\n \mid \bfP)$ when $\bfP$ is parent-independent restricted to the observed alleles also holds when $\bfP$ is reversible restricted to the observed alleles, provided that $\O \leq 3$. That formula fails to hold when $\O = 4$ and $\bfP$ is not parent-independent restricted to the observed alleles.

As there are four distinct DNA bases, our extension to the $\O=4$ case seems natural.  A more interesting reason is as follows:  In multi-locus models with finite recombination rates, no closed-form sampling formula is known, even for the simplest case of two loci with either infinite-alleles or finite-alleles PIM models.  However, recently a new framework  based on asymptotic series has  been developed \cite{jen:son:2009:G, jen:son:2010, jen:son:2011:aap,bha:son:2011} to derive useful closed-form results when the recombination rate is moderate to large.  The main idea behind that research is to perform an asymptotic expansion of the sampling probability in inverse powers of the recombination rate.  We note that our one-locus sampling formula for the $\O=4$ case provides an accurate approximation of the sampling probability for a completely linked (i.e., with zero recombination rate) pair of loci with two observed alleles at each locus (as is typical in single-nucleotide polymorphism data).  Hence, 
our work serves as a starting point for finding approximate two-locus sampling formulas when the recombination rate is small,
complementary to the earlier work \cite{jen:son:2009:G, jen:son:2010, jen:son:2011:aap,bha:son:2011} for large recombination rates.  We leave this problem for future research.

We remark that, for a given sample configuration $\n$ and fixed parameters $\theta$ and $\bfP$, the exact sampling probability $q(\n\mid\theta,\bfP)$ can be found numerically by solving a system of coupled linear equations in $O(|\n|^K)$ variables, where $|\n|$ denotes the total sample size and $K$ denotes the number of allele types in the assumed model.  One of the main motivations of our work is to remedy this high computational complexity.  Evaluating our closed-form approximations is much more efficient, in both time and space complexity.

The rest of this paper is structured as follows. In \sref{sec:model}, we lay out the model and notation used throughout the paper.  In \sref{sec:main_results}, we summarize our main closed-form sampling formulas, which we prove in \sref{sec:urn_proc} using martingale arguments and an urn construction.
Numerical experiments demonstrating the usefulness of our approximate sampling formulas are provided in \sref{sec:numerical_accuracy}.

\section{Model and notation \label{sec:model}}

We consider Kingman's coalescent with a $K$-allelic recurrent mutation model specified by the population-scaled mutation rate $\theta/2$ and ergodic transition matrix $\bfP$, where $P_{ji}$ denotes the probability of allele $j$ mutating to allele $i$ forward in time given that a mutation occurs. The stationary distribution of $\bfP$ is denoted by $\bo{\pi}=(\pi_1,\ldots,\pi_K)$.

The following definitions will be used throughout:

\begin{definition}[$\n$, sample configuration]
A sample of individuals is denoted by $\n = (n_i)_{i \in [K]}$, where $n_i\in \bbZ_{\geq 0}$ denotes the number of individuals in the sample with allele $i$. The size $|\n|$ of the sample $\n$ is denoted by the same letter in non-bold-face,  $n$. 
For notational convenience, we use $\e_i$ to denote the sample configuration with a single individual of type $i$ and write $\n = n_1 \e_1 + \cdots + n_K \e_K$.   
For a subset $S \subseteq [K]$, we define $\n_S = \sum_{i \in S} n_i \e_i$ and $n_S = |\n_S|$.
\end{definition}

\begin{definition}[$\Oset$, observed allele types]
Given a sample $\n$, let $\Oset\subseteq[K]$ denote the set of observed allele types; i.e., $\Oset = \{i\in [K] \mid n_i > 0\}$.   The number of observed allele types is denoted by $\O$.
\end{definition}

\noindent
When the indices $h, i, j, k$ and $l$ are used in indefinite summations or products, they are assumed to range over $\Oset$, unless stated  otherwise.

By exchangeability, the probability of any \emph{ordered} sample with configuration $\n$ is invariant under all permutations of the sampling order.  We use $q(\n\mid \theta,\bfP)$ to denote the stationary sampling probability of any particular ordered sample with configuration $\n$.
From the standard coalescent arguments \cite{gri:tav:1994a,gri:tav:1994b}, it can be deduced that  $q(\n \mid \theta,\bfP)$ is the unique solution to the recursion
\begin{equation}
	n(n-1+\theta) q(\n\mid \theta,\bfP) = \sum_{i} n_i (n_i-1) q(\n-\e_i\mid \theta,\bfP) + \theta \sum_{i,j}\, P_{ji}\, n_i\, q(\n-\e_i+\e_j\mid \theta,\bfP),
	\label{eq:main_rec}
\end{equation}
with boundary conditions
\begin{equation}
q(\e_i \mid \theta,\bfP) = \pi_i, \text{ for all $i \in [K]$.}
\end{equation}
If $\bfP$ is irreducible when restricted to the observed alleles $\Oset$, then by unwinding recursion \eqref{eq:main_rec}, it can be seen that $\O-1$ is the smallest power of $\theta$ with a non-vanishing coefficient in the Taylor series expansion of $q(\n \mid \theta, \bfP)$ about $\theta=0$. 
Intuitively, for a sample with $m$ distinct observed alleles, the coefficient of $\theta^{m-1}$ in the Taylor expansion corresponds to the total probability of coalescent genealogies with the most parsimonious number (i.e., $m - 1$) of mutations.
That $\bfP$ is irreducible when restricted to $\Oset$ is a sufficient (but not necessary) condition for the existence of such a parsimonious genealogy for sample $\n$.

Letting $\qt(\n \mid \bfP)$ denote the coefficient of $\theta^{\O-1}$ in the Taylor expansion, $q(\n \mid \theta, \bfP)$ can be written as in
\eqref{eq:q_expansion}.
For simplicity, in what follows we simply write $q(\n)$ and $\qt(\n)$ instead of $q(\n\mid\theta,\bfP)$ and $\qt(\n\mid\bfP)$, respectively.

We now introduce some notation used throughout the paper.
For a sample configuration $\n$, we define the combinatorial quantity $\Lambda(\n)$ as 
\begin{equation}
\Lambda(\n) = {\prod_{i\in\Oset} (n_i - 1)! \over (n-1)!}.
\label{eq:Lambda}
\end{equation}
For $k \in \bbZ_{\geq 0}$, the $k$th \emph{falling} factorial of $x$ (denoted $\fall{x}{k}$) and the $k$th \emph{rising} factorial of $x$ (denoted $\rise{x}{k}$) are defined as
\begin{align*}
	\fall{x}{k} {}={}& x(x-1)\cdots (x-k+1),\\
	\rise{x}{k} {}={}& x(x+1)\cdots (x+k-1),
\end{align*}
with $\fall{x}{0}=\rise{x}{0}=1$.
The $k$th harmonic number $H_k$ is defined as
\[
H_k = 1 + \frac{1}{2} + \cdots + \frac{1}{k},
\]
with $H_0 = 0$.
Given a sample configuration $\n=(n_1,\ldots,n_K)$, 
a $K$-tuple $\m=(m_1,\ldots,m_K)$ satisfying
$\bfzero\preceq\m\prec\n$ means $0\leq m_i < n_i$ for all $i \in \Oset$ and $m_i=0$ for all $i \notin \Oset$, while
$\bfzero\prec\m\preceq\n$ means $0< m_i \leq n_i$ for all $i \in \Oset$ and $m_i=0$ for all $i \notin \Oset$. Also,
$\bfzero\preceq\m\preceq\n$ denotes $0 \leq m_i \leq n_i$ for all $i \in [K]$.

\section{A summary of closed-form results for $\qt(\n)$ \label{sec:main_results}}
In the case of $\O=1$, it is easy to see that $\qt(\n)=\pi_i$ for $\n = n \e_i$.
In this paper, we derive closed-form expressions for the leading order coefficient $\qt(\n)$ when $\O \leq 3$ and $\bfP$ is an arbitrary mutation transition matrix that is irreducible when restricted to the observed alleles $\Oset$; and also when $\O = 4$, and $\bfP$ is irreducible and reversible when restricted to  $\Oset$ (i.e., $\pi_i P_{ij} = \pi_j P_{ji}$ for all $i, j \in \Oset$). 
These closed-form results are summarized below.


\begin{theorem}\label{thm:O2}
	For $\O = 2$ and  $\bfP$ an arbitrary mutation transition matrix that is irreducible when restricted to $\Oset$, $\qt(\n)$ is given by 
	\[
	\qt(\n) = \Lambda(\n) \sum_{i, j\in \Oset: i \neq j} {n_j \over n} \pi_j P_{ji}. 
	\]
\end{theorem}


\begin{theorem}\label{thm:O3}
	For $\O = 3$ and $\bfP$ an arbitrary mutation transition matrix that is irreducible when restricted to $\Oset$, $\qt(\n)$ is given by 
	\begin{align*}
	\qt(\n) {}={}& \Lambda(\n) \sum_{\text{\rm distinct}\, i, j, k \in\Oset} \Bigg\{ \pi_j P_{ji} P_{jk} \bigg[ {\fall{n_j}{2} \over n (n_j + n_k - 1)} - {n_i n_j \over n (n_i + n_k)} - 2 {n_i n_j n_k \over n \fall{n_j + n_k}{2}}  \\
			& \phantom{\Lambda(\n) \left[ \sum_{i, j, k\, \text{\rm distinct}} \pi_j P_{ji} P_{jk} \right.}  + 2 {n_i n_j n_k \over \fall{n_j + n_k + 1}{3}} (H_n - H_{n_i - 1}) \bigg] \\
	        & \phantom{\Lambda(\n) + \sum_{\text{\rm distinct}\, i, j, k \in\Oset}} + \pi_k P_{kj} P_{ji} \bigg[ {n_j n_k \over n (n_j + n_k - 1)} + 2 {n_i n_j n_k \over n \fall{n_j + n_k}{2}}  \\
			& \phantom{\Lambda(\n) \left[ \sum_{i, j, k\, \text{\rm distinct}} \pi_k P_{kj} P_{ji} \right.}  - 2 {n_i n_j n_k \over \fall{n_j + n_k + 1}{3}} (H_n - H_{n_i - 1}) \bigg] \Bigg\}.
	\end{align*}
\end{theorem}

\begin{corrol}\label{cor:O3reversible}
	Suppose $\O=3$ with sample configuration $\n=n_a \e_a + n_b \e_b + n_c \e_c$, where $a,b,c$ are distinct alleles in $[K]$.  If the mutation transition matrix $\bfP$ is reversible and irreducible when restricted to the observed alleles $\Oset$, $\qt(\n)$ is given by
	\[
	\qt(\n) = \Lambda(\n) \left( {n_a \over n} \pi_a P_{ab} P_{ac} + {n_b \over n} \pi_b P_{ba} P_{bc} + {n_c \over n} \pi_c P_{ca} P_{cb} \right).
	\]
\end{corrol}


\begin{theorem}\label{thm:O4}
	For $\O = 4$, if the mutation transition matrix $\bfP$ is reversible and irreducible when restricted to the observed alleles $\Oset$, then $\qt(\n)$ is given by 
	\begin{align*}
	\qt(\n) = \Lambda(\n) \sum_{ \text{\rm distinct}\, i, j, k, l\in\Oset} \left[
	\pi_i P_{ij} P_{ik} P_{il} \gamma(\n, i, j, k, l) + \pi_i P_{ij} P_{ik} P_{jl} \delta(\n, i, j, k, l)\right],
	\end{align*}
	where 
	\begin{align*}
\gamma(\n, i, j, k, l) {}={}& {n_i \over n} \Bigg\{ \bigg[ {n_i - 1 \over 2 (n_i + n_j + n_k - 1)} - {2 n_j n_l \over \fall{n_i + n_j + n_k}{2}} \bigg] + {n_l \over 2(n_j + n_k + n_l)} \\
	&    \phantom{{n_i \over n} \left[ \right.}- \bigg[ {n_l (n_i - 1) \over (n_k + n_l)(n_i + n_j - 1)} - {2 n_j n_l \over \fall{n_i + n_j}{2}} \bigg] \Bigg\} \\
	& \hspace{-5mm} + {2 n_i n_j n_l \over \fall{n_i + n_j + n_k + 1}{3}} (H_n - H_{n_l - 1}) - {2 n_i n_j n_l \over \fall{n_i + n_j + 1}{3}} (H_n - H_{n_k + n_l - 1}), 
	\end{align*}
and
	\begin{align*}
	{\delta(\n, i, j, k, l)} {}={}& {n_i \over n} \Bigg\{ \bigg[{n_j \over {n_i + n_j + n_k - 1}} + {2 n_j n_l \over \fall{n_i + n_j + n_k}{2}} \bigg]\\
	&\phantom{{n_i \over n} \Bigg\{} - \bigg[ {n_j n_l \over (n_k + n_l)(n_i + n_j - 1)} + {2 n_j n_l \over \fall{n_i + n_j}{2}} \bigg] \Bigg\}	\\
	& \hspace{-5mm} - {2 n_i n_j n_l \over \fall{n_i + n_j + n_k + 1}{3}} (H_n - H_{n_l - 1}) + {2 n_i n_j n_l \over \fall{n_i + n_j + 1}{3}} (H_n - H_{n_k + n_l - 1}).
	\end{align*}
\end{theorem}

\section{Proofs of the main results \label{sec:urn_proc}}

In this section, we construct an urn process to derive the closed-form formulas for $\qt(\n)$ mentioned in the previous section.  We use the urn process to decompose $\qt(\n)$ into a sum-product of two vectors, one which  depends
only on the sample configuration $\n$ and the other which  depends only on the mutation transition matrix $\bfP$. Using this decomposition, we show that $\qt(\n)$ corresponds to the probability of a certain event in the urn process.

Throughout, we use $\rt(\n)$ to denote the following rescaled version of $\qt(n)$:
\begin{align}
\rt(\n) &= {\qt(\n) \over \Lambda(\n)},
\label{eq:qt_rt}
\end{align}
where $\Lambda(\n)$ is the combinatorial coefficient defined in \eqref{eq:Lambda}.

\subsection{Description of the urn process}

Let $\n$ be the sample configuration of interest.
We have an urn with $n$ balls, $n_i$ of which have color $i$. We remove balls one at a time uniformly at random until there are no more balls in the urn. However, whenever we ``kill'' a color (i.e., remove the last ball of that color), we add back a ball of a different color. We do this by picking another ball from the urn, copying it, and returning both copies to the urn. Note that when we kill the last color, we do not add any balls back, since there are no more colors to choose from.

Suppose that when we kill color $i$, we add back a ball of color $j$. We then call $j$ the \emph{parent} of $i$, and call the last surviving color the \emph{root}. This generates a rooted tree whose vertices consist of the $\O$ observed colors (alleles).

Let $T$ be any rooted tree on $\Oset$. We denote the probability of generating $T$ under the above process as $\Pn(T)$. Let $E(T)$ be the edge set of $T$, and let $\rho(T)$ denote the root vertex of $T$. By convention, we draw edges as pointing away from the root, so the edge $(j \to i)$ indicates that $j$ is the parent of $i$.

The main idea of this section is that to compute $\qt(\n)$, it is enough to compute $\Pn(T)$ for each $T$. In particular, we prove the following theorem in \sref{sec:proof_thm4}:
\begin{theorem}\label{treesum}
  Recall that for a transition matrix $\bfP$ that is irreducible when restricted to $\Oset$, $\qt(\n)$ denotes the first nonzero coefficient in the Taylor expansion \eqref{eq:q_expansion} of $q(\n)$ about $\theta=0$.
Given a rooted tree $T$ described above, define $f_{\bfP}(T)$ as
\begin{equation*}
  f_{\bfP}(T) = \pi_{\rho(T)} \prod_{(j \to i) \in E(T)} P_{ji}.
\end{equation*}
Then, the quantity  $\rt(\n) = \qt(\n)/\Lambda(\n)$ is given by
  \begin{equation}
  \rt(\n) = \sum_T \Pn(T) f_{\bfP}(T) =  \En[f_{\bfP}(T)], 
  \label{eq:rtn_form}
  \end{equation}
where the sum is taken over all rooted trees $T$ with $\O$ vertices bijectively labeled by $\Oset$.
That is,  $\rt(\n)$ is the expectation of $f_{\bfP}(T)$ under the above process.
\end{theorem}

Note that we can view $f_{\bfP}(T)$ as a probability as well. In particular, suppose we relabel the vertices of $T$ as follows: we assign a new label from $[K]$ to $\rho(T)$ according to the stationary distribution $\pi$, and for each edge in $T$, we assign a new label to the child according to the new label of its parent and the transition matrix $\bfP$. Then $f_{\bfP}(T)$ is the probability that we assign the original labels to all the vertices, given that we drew $T$. That is, if $\C_{\Oset}$ is the event that we assign the original labels to all vertices, then
\begin{equation*}
  f_{\bfP}(T) = \bbP(\C_{\Oset} \mid T) = \pi_{\rho(T)} \prod_{(j \to i) \in E(T)} P_{ji}.
\end{equation*}

\noindent
This immediately leads to the following interpretation:
\begin{align}
    \rt(\n) = \sum_T \bbP(\C_{\Oset} \mid T) \Pn(T) = \Pn(\C_{\Oset}).
\label{eq:relabel} 
\end{align}
That is, $\rt(\n)$ is the unconditional probability that we correctly label all the alleles, if we use the urn process to generate a tree on the alleles and then use the tree to assign labels.

\subsection{An inductive proof of \thmref{treesum}}
\label{sec:proof_thm4}

In this subsection, we provide an inductive proof of \thmref{treesum}.
In \sref{mod_coal}, we provide an alternative proof based on a modified coalescent process which provides a more intuitive explanation for why the urn process works.

\begin{proof}[Proof of \thmref{treesum}]
Recall the recursion in \eqref{eq:main_rec}:
\begin{equation*}
	n(n-1+\theta) q(\n) = \sum_{i} n_i (n_i-1) q(\n-\e_i) + \theta \sum_{i,j}\, P_{ji}\, n_i\, q(\n-\e_i+\e_j).
\end{equation*}
Recall also that if $\bfP$ is irreducible when restricted to $\Oset$, $q(\n)$ has leading order power $\theta^{\O - 1}$ in its Taylor series. Hence we get the following recursion for $\qt(\n)$:
\begin{equation*}
  n(n-1) \qt(\n) = \sum_{i: n_i > 1} n_i (n_i - 1) \qt(\n - \e_i) + \sum_{i: n_i = 1} \sum_{j: j \neq i} P_{ji} n_i \qt(\n - \e_i + \e_j).
\end{equation*}
Plugging in $\qt(\n) = \Lambda(\n) \rt(\n)$ and simplifying gives us the following recursion for $\rt(\n)$:
\begin{equation}
  n(n-1) \rt(\n) = \sum_{i: n_i > 1} n_i (n - 1) \rt(\n - \e_i) + \sum_{i: n_i = 1} \sum_{j: j \neq i} P_{ji} n_j \rt(\n - \e_i + \e_j).
\label{eq:rt_recurs}
\end{equation}

A simple induction over $\O$ and $n$ shows
that this recursion has a unique solution
given the boundary conditions $\rt(\e_i)$.
So if we can show \eqref{eq:rtn_form} when $\O = n = 1$, and then show that $\sum_T \bbP(\C_{\Oset} \mid T) \Pn(T)$ satisfies the recursion \eqref{eq:rt_recurs}, then we will be done.
The base case is trivial: when $\Oset = \{a\}$, there is only one possible tree, $T = \{ a\}$, with $\Pn(T) = 1$ and $\bbP(\C_{\Oset} \mid T) = \pi_a = \lim_{\theta \to 0} q(\n) = \qt(\n) = \Lambda(\n) \rt(\n) = \rt(\n)$.

To show $\sum_T \bbP(\C_{\Oset} \mid T) \Pn(T)$ satisfies \eqref{eq:rt_recurs}, we start by giving recursions for $\Pn(T)$ and $\bbP(\C_{\Oset} \mid T)$. Let $z(i)$ be the parent of $i$ in $T$, and let $L(T)$ be the set of leafs of $T$ (where the root is not considered a leaf). Conditioning on the first event in the urn process gives us
\begin{equation}
  \Pn(T) = \sum_{i: n_i > 1} \frac{n_i}n \bbP_{\n - \e_i}(T)
  + \sum_{\substack{ i \in L(T): n_i = 1}} \frac{n_{z(i)}}{n (n-1)} \bbP_{\n - \e_i + \e_{z(i)}}(T \setminus \{i\}).
\label{eq:tree_recurs}
\end{equation}
Furthermore, if $i \in L(T)$, we have
\begin{equation}
 \bbP(\C_{\Oset} \mid T) =  P_{z(i),i} \bbP(\C_{\Oset \setminus \{i\}} \mid T \setminus \{i\}).
\label{eq:pct_recurs}
\end{equation}
Using \eqref{eq:tree_recurs} and \eqref{eq:pct_recurs}, and collecting terms, we arrive at 
\begin{align*}
\lefteqn{n (n-1) \sum_T \bbP(\C_{\Oset} \mid T) \Pn(T)}  \hspace*{1cm}&\\
=&  \sum_T \bbP(\C_{\Oset} \mid T) \Bigg[ \sum_{i: n_i > 1} n_i (n-1) \bbP_{\n - \e_i}(T)
  + \sum_{\substack{i \in L(T): n_i = 1}} n_{z(i)} \bbP_{\n - \e_i + \e_{z(i)}}(T \setminus \{i\})\Bigg]\\
=&  \sum_{i: n_i > 1} n_i (n-1) \sum_T \bbP_{\n - \e_i}(T) \bbP(\C_{\Oset} \mid T) \\
&{} + \sum_{\substack{i: n_i = 1}}\sum_{j:j\neq i}  P_{ji} n_j \sum_{T'} \bbP_{\n - \e_i + \e_{j}}(T') \bbP(\C_{\Oset \setminus \{i\}} \mid T'),
\end{align*}
where the sum over $T'$ is taken over all rooted trees with vertex set $\Oset \setminus \{i\}$. 
Hence, $\sum_T\bbP(\C_{\Oset}\mid T)\Pn(T)$ satisfies \eqref{eq:rt_recurs}.
\hfill{}\qed
\end{proof}

\subsection{Connection to the coalescent \label{mod_coal}}

In this subsection, we motivate our urn process by drawing a connection to the coalescent. We then use this connection with the coalescent to provide an alternate proof of \thmref{treesum}.

Let $\H$ be a history of mutation and coalescence events on $n$ labeled individuals, and let $q(\H)$ be the probability of $\H$. Then we have 
\begin{equation}
  q(\n) = \sum_{\text{$\H$ consistent with $\n$}} q(\H). \label{eq:sum_qH}
\end{equation}
It turns out that only histories with exactly $\O - 1$ mutations contribute to the leading order term of $q(\n)$; this is the observation also utilized in
\cite{jen:son:2011:tpb}.  Furthermore, each history of choices in our urn process corresponds with a genealogical history of $\O - 1$ mutations. This provides the basic intuition for why the urn sampling scheme works.

We start by providing a modified coalescent that generates a history $\H$ that is consistent with $\n$ and has exactly $\O - 1$ mutations. We then show that this modified coalescent is equivalent to our urn sampling process. Finally, we prove \thmref{treesum} by relating the modified coalescent with Kingman's coalescent.

Consider the following modified coalescent process
on our sample:
\begin{enumerate}
\item Select allele $i$ with probability $m_i/m$, where $\m$ is our current configuration of alleles.
\item If $m_i > 1$, choose a random pair in allele $i$ to coalesce (so $\m$ is replaced with $\m - \e_i$).
\item If $m_i = 1$, have the last individual of allele $i$ mutate to allele $j$ with probability $m_j/(m-1)$ (so $\m$ is replaced with $\m - \e_i + \e_j$).
\item Repeat steps 1 to 3 until all individuals have coalesced.
\end{enumerate}

It should be clear that the modified coalescent only generates histories with exactly $\O - 1$ mutations, since each mutation kills an allele permanently.

If we take an unordered view of our sample, then the modified coalescent is equivalent to the urn process, for they have the same initial configuration and transition probabilities between configurations. In particular, when $m_i > 1$ we move from $\m$ to $\m - \e_i$ with probability $m_i/m$, and when $m_i = 1$ we move from $\m$ to $\m - \e_i + \e_j$ with probability $m_j/\fall{m}{2}$. We generate trees on $\Oset$ by drawing an edge $(j \to i)$ whenever we make a transition from $\m$ to $\m - \e_i + \e_j$, i.e. whenever there is a mutation from $i$ to $j$.

We now give a proof of \thmref{treesum}, using the modified coalescent in place of the urn process:

\begin{proof}[Alternative proof of \thmref{treesum}]
Let $\H$ be a coalescent history with exactly $M$ mutations. Running time backwards from the present, we suppose that the $i$th mutation was from allele $u_i$ to allele $v_i$, and that the most recent common ancestor has allele $\rho$. We further suppose that $J_i$ is the total number of lineages at the time of the $i$th mutation. Then we have that
\begin{equation*}
  q(\H) = \pi_\rho \Big(\prod_{i = 1}^{M} P_{v_i u_i} \Big)
  \frac{\theta^{M}}{\prod_{i=1}^{M} J_i (\theta + J_i - 1)} \frac{2^{n-1}}{n! \fall{\theta + n - 1}{(n-1)}} ,
\end{equation*}
since the $i$th coalescence contributes probability $\frac{n-i}{n-i+\theta} {n -i + 1 \choose 2}^{-1} = \frac2{(n - i + 1) (n-i + \theta)}$,
and the $i$th mutation contributes probability $\frac{\theta P_{v_i u_i}}{J_i (J_i - 1 + \theta)}$.

Now, observe that
\begin{equation}
  \qt(\H) \equiv \lim_{\theta \to 0} \frac{q(\H)}{\theta^{M}} = \pi_{\rho} \Big(\prod_{i = 1}^{M} P_{v_i u_i} \Big) \frac{2^{n-1}}{n! (n-1)! \prod_{i=1}^{M} J_i (J_i - 1)}. \label{eq:qtH}
\end{equation}
Therefore, the Taylor series for $q(\H)$
has leading power $\theta^{M}$, with coefficient $\qt(\H)$.

Hence by \eqref{eq:sum_qH}, the Taylor series for $q(\n)$ has leading power $\theta^{\O - 1}$, and its leading coefficient is given by the sum of all $\qt(\H)$ such that $\H$ is consistent with $\n$ and has $\O - 1$ mutations.

For such an $\H$, let $\Pn(\H)$ be the probability of generating $\H$ under our modified coalescent.
Then we have that
\begin{equation}
  \Pn(\H) = \frac{2^{n-1}}{n! \prod_{k=1}^{\O} (n_k-1)! \prod_{i=1}^{\O - 1} J_i (J_i - 1)}. \label{eq:PnH}
\end{equation}
To see this, note that if our current sample is $\m$, the probability that the next event is a coalescence on allele $i$ with $m_i > 1$ is
\begin{equation*}
  \frac{m_i}{m} \frac{2}{m_i (m_i - 1)} = \frac{2}{m (m_i - 1)},
\end{equation*}
and if $m_i = 1$, the probability that the next event is a mutation from allele $i$ to allele $j$ (where $j \neq i$) is
\begin{equation*}
  \frac{m_j}{m (m-1)}.
\end{equation*}
Multiplying the probabilities of the mutation and coalescence events in $\H$, and noting that the numerator of each mutation term cancels with the denominator of a future coalescence term, yields the equation \eqref{eq:PnH}.

Combining \eqref{eq:qtH} with \eqref{eq:PnH} yields
\begin{equation*}
  \qt(\H) = \Lambda(\n) \pi_{\rho} \Big(\prod_{i = 1}^{\O - 1} P_{v_i u_i} \Big) \Pn(\H)
\end{equation*}

Now let $\mathscr{T}(\H)$ be the resulting tree on $\Oset$ if
we draw an edge $(j \to i)$ when allele $i$ mutates to allele $j$.
Then we have
\begin{align*}
  \qt(\n) &= \sum_{\substack{\text{$\H$ consistent with $\n$} \\ \text{$\H$ has $\O - 1$ mutations}}} \qt(\H) \\
  &= \Lambda(\n) \sum_{T} \pi_{\rho(T)} \Big(\prod_{(j \to i) \in T} P_{ji} \Big) \Big(\sum_{\H: \mathscr{T}(\H) = T} \Pn(\H) \Big) \\
  &= \Lambda(\n) \sum_{T} \pi_{\rho(T)} \Big(\prod_{(j \to i) \in T} P_{ji} \Big) \Pn(T) \\
  &= \Lambda(\n) \sum_{T} f_{\bfP}(T) \Pn(T),
\end{align*}
and hence
\[
\rt(\n) = \sum_T f_{\bfP}(T) \Pn(T),
\]
as needed. \hfill{}\qed
\end{proof}

\subsection{A martingale property \label{martingale}}

Here, we prove \thmref{thm:O2} and \corref{cor:O3reversible} by using martingales to compute $\Pn(T)$ for $\Oset = \{a,b\}$, and for $\Oset = \{a,b,c\}$ when $\bfP$ is reversible when restricted to $\Oset$.
We run time as follows: whenever we remove a ball in the urn process, count this as one time step. If in doing so, we kill a color, count the adding of another ball as a separate time step.

Let $\F_t$ be the $\sigma$-algebra generated by all sequences of choices up to time $t$. Let $X_t$ be the proportion of balls that have color $a$ at time $t$; so $X_0 = n_a/n$. It is easy to check that $\{X_t\}$ is a martingale with respect to $\{\F_t\}$: Suppose that $\m$ is the remaining sample after time $t-1$, and we are deleting a ball at time $t$. Then,
\begin{align*}
  \bbE[X_t \mid \F_{t-1}] &= \frac{m_a}{m} \frac{m_a - 1}{m - 1} + \sum_{i \neq a} \frac{m_i}{m} \frac{m_a}{m - 1} = \frac{m_a}{m} = X_{t-1}.
\end{align*}
On the other hand, if we are adding a ball at time $t$, then
\begin{align*}
  \bbE[X_t \mid \F_{t-1}] &= \frac{m_a}{m} \frac{m_a + 1}{m + 1} + \sum_{i \neq a} \frac{m_i}{m} \frac{m_a}{m + 1} = \frac{m_a}{m} = X_{t-1}.
\end{align*}
So, $\{(X_t,\F_t),t\geq 0\}$ is a martingale.

\begin{proof}[Proof of \thmref{thm:O2}]
Suppose $\Oset = \{a,b\}$. Let $T$ be the tree whose vertex set is $\Oset$, with $a$ being the root. Let $\tau$ be the the first time we kill a color. Noting that $\tau$ is a stopping time, we obtain
\begin{align*}
  \Pn(T) &= \bbE[ \Pn(T \mid \F_{\tau}) ] \\
  &= \bbE[ \I(\text{Color $a$ is the last remaining at time $\tau$})] \\
  &= \bbE[X_{\tau}] = \bbE[X_{0}]=\frac{n_a}n.
\end{align*}
Therefore, by \thmref{treesum},
\begin{equation*}
  \qt(\n) = \Lambda(\n) \Big( \frac{n_a}n \pi_a P_{ab} + \frac{n_b}n \pi_b P_{ba} \Big).
\tag*{\hfill{}\qed}
\end{equation*}
\end{proof}

\begin{proof}[Proof of \corref{cor:O3reversible}]
Suppose $\Oset = \{a,b,c\}$ and $\bfP$ is reversible  when restricted to $\Oset$. Note that $\bbP(\C_{\Oset} \mid T)$ does not depend on how $T$ is rooted, for by reversibility we can move the root around by
\begin{equation*}
  \pi_\rho P_{\rho k} = \pi_k P_{k \rho}, \qquad \forall k \in \Oset, k \neq \rho.
\end{equation*}
Therefore, we redefine $\Pn(T)$ to be the probability of drawing the undirected tree $T$. We still have $\rt(\n) = \sum_T \bbP(\C_{\Oset} \mid T) \Pn(T)$, but now the sum is taken over undirected $T$.
Now let $T$ be the tree on $\{a,b,c\}$ whose interior vertex is $a$. We draw $T$ if and only if $a$ is chosen as the parent of the first color that we kill.  So, letting $\tau$ be the first killing time and noting $X_\tau = \Pn(T\mid\F_\tau)$, we have
\begin{equation*}
  \Pn(T) = \bbE[\Pn(T \mid \F_\tau)] = \bbE[X_\tau] = \bbE[X_0]=\frac{n_a}{n}.
\end{equation*}
Therefore, by \thmref{treesum},
\begin{equation*}
  \qt(\n) = \Lambda(\n) \Big( \frac{n_a}n \pi_a P_{ab} P_{ac}
  + \frac{n_b}n \pi_b P_{ba} P_{bc}
  + \frac{n_c}n \pi_c P_{ca} P_{cb} \Big).
\tag*{\hfill{}\qed}
\end{equation*}
\end{proof}

\subsection{A recursion for $\rt(\n)$}
In this section, we derive a recursion for $\rt(\n)$ which will be useful for deriving closed-form formulas for $\qt(\n)$ when $\O = 3, 4$.
Given a sample configuration $\n$ and a subsample $\m$, define the expression ${\n \choose \m}$ as
\[
{\n \choose \m} = \prod_{i \in \Oset} {n_i \choose m_i}.
\]
The following proposition provides a recursion relating $\rt(\n)$ to $\rt(\m)$ where $\Om = \O - 1$.

\begin{propos}\label{prop:rt_recursion}
Suppose $\bfP$ is irreducible when restricted to $\Oset$ and
let $\theta^{\O-1}\qt(\n)=\theta^{\O-1}\Lambda(\n)\rt(\n)$ denote the leading order term in the Taylor expansion \eqref{eq:q_expansion} of $q(\n)$ about $\theta=0$.  Then,
 $\rt(\n)$ for $\O > 1$ satisfies the recursion
    \begin{align}\label{eq:rt_recursion}
      \rt(\n) = \sum_{i, j \in \Oset: i \neq j} P_{ji} \sum_{\substack{\bfzero \prec \m \preceq \n: \\ m_i = 1}} {{\n \choose \m} \over {n \choose m}} {m_j \rt(\m - \e_i + \e_j) \over m (m-1)}, 
    \end{align}
     with boundary conditions
    \begin{align}\label{eq:rt0}
    \rt(\n) = \pi_a,
    \end{align}
    for all sample configurations $\n = n_a \e_a$, where  $a \in [K]$.
\end{propos}

\begin{proof}[Proof of  \propref{prop:rt_recursion}]
We can derive this recursion from the urn process as follows.	
Let $D_{ij}(\m)$ be the event where the first killing replaces a ball of color $i$ with a ball of color $j$, and where $\m$ is the (unordered) configuration immediately before this killing. 
Then for any event $A$,
\begin{equation}
  \Pn(A) = \sum_{i, j \neq i} \sum_{\substack{\bfzero \prec \m \preceq \n: \\ m_i = 1}} \Pn(D_{ij}(\m)) \Pn(A \mid D_{ij}(\m)) \label{eq:urn_recurs}
\end{equation}
where we use the fact that $\Pn(D_{ij}(\m)) = 0$ if $m_i \neq 1$ or
$m_j = 0$ for any $j \in \Oset$.

We compute $\Pn(D_{ij}(\m))$ when
$\m \succ \bfzero$ and $m_i = 1$.
The probability that $\m$ is the remaining configuration
after $n - m$ draws is
\begin{equation*}
  \frac{(n-m)!}{\prod_k (n_k - m_k)!} \frac{\prod_k \fall{n_k}{n_k - m_k}}{\fall{n}{n-m}}
  = \frac{{\n \choose \m}}{{n \choose m}}.
\end{equation*}
To see this, note that the first term is the number of ways we can make $n-m$ draws that result in the configuration $\m$, and the second term is the probability of each such sequence of draws.

When our current configuration is $\m$ with $m_i = 1$, the probability that on the next draw we replace the last ball of color $i$ with a ball of color $j$ is $m_j/\fall{m}{2}$. Hence we get that
\begin{equation*}
  \Pn(D_{ij}(\m)) = \frac{{\n \choose \m}}{{n \choose m}} \frac{m_j}{m (m-1)}.
\end{equation*}
when $\m \succ \bfzero$ and $m_i = 1$.

Plugging this into \eqref{eq:urn_recurs} yields
\begin{equation}
  \Pn(A) = \sum_{i, j \neq i} \sum_{\substack{\bfzero \prec \m \preceq \n: \\ m_i = 1}} \frac{{\n \choose \m}}{{n \choose m}} \frac{m_j}{m(m-1)} \Pn(A \mid D_{ij}(\m)). \label{eq:urn_recurs2}  
\end{equation}

Now recall from \eqref{eq:relabel} that $\rt(\n) = \Pn(\C_{\Oset})$. That is, $\rt(\n)$ is the probability that we assign the original labels to all alleles, if we use the urn process to generate a tree on $\Oset$ and then use the tree to assign new labels to the alleles. Note that
\begin{equation*}
  \bbP(\C_{\Oset} \mid D_{ij}(\m)) = P_{ji} \bbP_{\m - \e_i + \e_j}(\C_{\Oset \setminus \{i\}})
  = P_{ji} \rt(\m - \e_i + \e_j),
\end{equation*}
since we need to use the urn process with sample $\m - \e_i + \e_j$ to correctly relabel $\Oset \setminus \{i\}$, and then assign the correct label to $\{i\}$ with probability $P_{ji}$. Plugging this into \eqref{eq:urn_recurs2} with $A = \C_{\Oset}$ yields the desired recursion,
\begin{equation*}
  \rt(\n)  = \sum_{i, j \neq i} P_{ji} \sum_{\substack{\bfzero \prec \m \preceq \n: \\ m_i = 1}} {{\n \choose \m} \over {n \choose m}} {m_j \rt(\m - \e_i + \e_j) \over m (m-1)} .  
\tag*{\hfill{}\qed}
\end{equation*}
\end{proof}

In the next two subsections, we use the recursion in \propref{prop:rt_recursion} to provide proofs of \thmref{thm:O3} and \thmref{thm:O4}.

\subsection{Proof of \thmref{thm:O3}  ($\O=3$) \label{sec:combinatorial_proof_O3}}
For $\O = 3$, the following expression for $\rt(\n)$ can be derived using \propref{prop:rt_recursion}:
\begin{align}
\rt(\n) =& \sum_{i, j \neq i} P_{ji} \sum_{\substack{\bfzero \prec \m \preceq \n: \\ m_i = 1}} {{\n \choose \m} \over {n \choose m}} {m_j \rt(\m - \e_i + \e_j) \over m (m-1)}    \nonumber \\
        =& \sum_{i, j \neq i} P_{ji} \sum_{\substack{\bfzero \prec \m \preceq \n: \\ m_i = 1}} {{\n \choose \m} \over {n \choose m}} {m_j \over m (m-1)}
\sum_{\substack{k, l:\\ l \neq k \text{ and } k, l \neq i}} {m_k + \delta_{j,k} \over m} \pi_k P_{kl}   \nonumber \\
        =& \sum_{i, j \neq i} P_{ji} \sum_{\substack{\bfzero \prec \m \preceq \n: \\ m_i = 1}} \left\{ \rule{0mm}{7mm}
        {{\n \choose \m} \over {n \choose m}} {1 \over m^2 (m-1)} \right. \nonumber \\
         & \hspace{2.5cm}\times \left. \left[ \sum_{l: l \neq i, j} m_j(m_j+1) \pi_{j} P_{jl} + \sum_{k: k \neq i, j} m_j m_k \pi_k P_{kj} \right]\right\}  \nonumber \\
        =& \sum_{i, j \neq i} P_{ji} \sum_{m=3}^{n} \sum_{\substack{\bfzero \prec \m \preceq \n:\\ m_i = 1, |\m|=m}} \left\{ \rule{0mm}{7mm}
        {{\n \choose \m} \over {n \choose m}} {1 \over m^2 (m-1)}  \right.  \nonumber \\
        &\hspace{2.5cm}\times  \left.\left[ \sum_{k: k \neq i, j} m_j(m_j+1) \pi_{j} P_{jk} + \sum_{k: k \neq i, j} m_j m_k \pi_k P_{kj} \right]\right\}  \nonumber \\        
 		=& \sum_{i, j, k \text{ distinct}}  \sum_{m=3}^{n} \sum_{\substack{\bfzero \prec \m \preceq \n:\\ m_i = 1,|\m|=m}} {{\n \choose \m} \over {n \choose m}} 
{\pi_j P_{ji} P_{jk}  m_j (m_j + 1) + \pi_k P_{kj} P_{ji}m_j m_k \over m^2 (m-1)}, \label{eq:rt_O3_1}
\end{align}
where in the second equality, the formula from \thmref{thm:O2} is used,  noting that $|\mathcal{O}_{\m-\e_i+\e_j}|=2$.
If we define the quantities $\alpha(\n, i, j, k)$ and $\beta(\n, i, j, k)$ as 
\begin{align}
\alpha(\n, i, j, k) = \sum_{m = 3}^{n} {1 \over m^2 (m-1)} \sum_{\substack{\bfzero \prec \m \preceq \n: \\ m_i = 1, |\m|=m}} {{\n \choose \m} \over {n \choose m}} m_j (m_j + 1),	\label{eq:alpha_O3}
\end{align}
and 
\[
\beta(\n, i, j, k) = \sum_{m = 3}^{n} {1 \over m^2 (m-1)} \sum_{\substack{\bfzero \prec \m \preceq \n: \\ m_i = 1, |\m|=m}} {{\n \choose \m} \over {n \choose m}} m_j m_k,
\]
then \eqref{eq:rt_O3_1} can be rewritten as
\begin{align}
\rt(\n) = \sum_{i, j, k \text{ distinct}} \pi_j P_{ji} P_{jk} \alpha(\n, i, j, k) + \sum_{i, j, k \text{ distinct}} \pi_k P_{kj} P_{ji} \beta(\n, i, j, k).	\label{eq:rt_O3_breakup}
\end{align}
Now consider $\alpha(\n, i, j, k)$ defined by \eqref{eq:alpha_O3}.
We can remove the restriction in the inner sum that $m_i = 1$ by defining $\m' = \m - \e_i$, and so $|\m'| = m-1$. Also, since $j \neq i$ in \eqref{eq:rt_O3_breakup}, $m'_j = m_j$. Making this change of variables from $\m$ to $\m'$ in the inner sum of \eqref{eq:alpha_O3}, we get
\begin{align}
\sum_{\substack{\bfzero \prec \m \preceq \n: \\ m_i = 1, |\m| = m}} {{\n \choose \m} \over {n \choose m}} m_j (m_j + 1) = 
{{n - n_i \choose m - 1} \over {n \choose m}} n_i \sum_{\substack{\bfzero \prec \m' \preceq \n - n_i \e_i: \\ |\m'|=m-1}} {{\n - n_i \e_i \choose \m'} \over {n - n_i \choose m - 1}} m'_j (m'_j + 1). \label{eq:rt_O3_1_1}
\end{align}
Using identity \eqref{id:second_moment_tweak} in \factref{fact:restricted_moment} of the Appendix, the summation over $\m'$ in \eqref{eq:rt_O3_1_1} can be written as
\begin{align}
\lefteqn{\sum_{\substack{\bfzero \prec \m' \preceq \n - n_i \e_i: \\ |\m'|=m-1}} {{\n - n_i \e_i \choose \m'} \over {n - n_i \choose m - 1}} m'_j (m'_j + 1)}\hspace{20mm}	\nonumber \\
={} & \sum_{\substack{T \subseteq [L]: \\ i, j \notin T}} (-1)^{|T|} \left[ {\fall{n_j}{2} \fall{m-1}{2} \over \fall{n-n_i-n_T}{2}} + {2 n_j (m-1) \over n - n_i - n_T} \right] {{n - n_i - n_T \choose m-1} \over {n - n_i \choose m-1}}	\label{eq:rt_O3_1_2}
\end{align}
The only sets $T$ satisfying the conditions in the summation in \eqref{eq:rt_O3_1_2} are $T = \varnothing$ and $T = \{k\}$. Hence, substituting \eqref{eq:rt_O3_1_1} and \eqref{eq:rt_O3_1_2} in \eqref{eq:alpha_O3}, we have
\begin{align}
\alpha(\n, i, j, k)  
&={}  \sum_{m = 3}^{n} {1 \over m^2 (m-1)} \sum_{\substack{\bfzero \prec \m \preceq \n: \\ m_i = 1, |\m|=m}} {{\n \choose \m} \over {n \choose m}} m_j (m_j + 1)   \nonumber \\
&={} \sum_{m=3}^{n} {1 \over m^2 (m-1)} {{n - n_i \choose m - 1} \over {n \choose m}} n_i
    \sum_{\substack{\bfzero \prec \m' \preceq \n - n_i \e_i: \\ |\m'|=m-1}} {{\n - n_i \e_i \choose \m'} \over {n - n_i \choose m-1}} m'_j(m'_j+1)  \nonumber \\    
&={} \sum_{m=3}^{n} {n_i {n - n_i \choose m - 1} \over m^2 (m-1) {n \choose m}} \left\{
    {{n_j + n_k \choose m - 1} \over {n-n_i \choose m-1}} \left[ {\fall{n_j}{2} \over \fall{n_j + n_k}{2}} \fall{m-1}{2} + 2 {n_j(m-1) \over n_j + n_k}  \right]  \right. \nonumber \\
&   \hspace{1.6in} \left. - {{n_j \choose m-1} \over {n-n_i \choose m-1}} \left[ {\fall{n_j}{2} \over \fall{n_j}{2}} \fall{m-1}{2} + 2 {n_j \over n_j} (m-1) \right]  \right\}	\nonumber \\
&={} \sum_{m=3}^{n} {n_i \over m^2 (m-1)} \left\{
    {{n_j + n_k \choose m - 1} \over {n \choose m}} \left[ {\fall{n_j}{2} \over \fall{n_j + n_k}{2}} \fall{m-1}{2} + 2 {n_j  (m-1)\over n_j + n_k} \right]	\right. \nonumber \\
&   \hspace{3cm} \left. - {{n_j \choose m-1} \over {n \choose m}} m (m - 1) \right\}   \nonumber \\
&={} \sum_{m=1}^{n} {n_i \over n} \left\{ {{n_j + n_k \choose m} \over {n-1 \choose m}} 
     \left[ {\fall{n_j}{2} \over \fall{n_j + n_k}{2}} {m-1 \over m+1} + 2 {n_j \over n_j + n_k} {1 \over m+1} \right]
          - {{n_j \choose m} \over {n-1 \choose m}} \right\}. \label{eq:rt_O3_2}
\end{align}
Applying Facts~\ref{fact:ratio} and \ref{fact:ratio2} in the Appendix  to \eqref{eq:rt_O3_2} yields
\begin{align}
\alpha(\n, i, j, k) 
&= \sum_{m=1}^{n} {n_i \over n} \left[ 
     {\fall{n_j}{2} \over \fall{n_j + n_k}{2}} {{n_j + n_k \choose m} \over {n-1 \choose m}} - {{n_j \choose m} \over {n-1 \choose m}}
   + {2 n_j n_k \over \fall{n_j + n_k}{2}} {{n_j + n_k \choose m} \over {n-1 \choose m}} {1 \over m+1} \right]  \nonumber \\
&= {n_i \over n} \left\{ {\fall{n_j}{2} \over \fall{n_j + n_k}{2}} {n_j + n_k \over n_i} - {n_j \over n_i + n_k} \right. \nonumber \\
&  \hspace{0.8cm} \left. + 2 {n_j n_k \over \fall{n_j + n_k}{2}} \left[{n \over n_j + n_k + 1} (H_n - H_{n_i - 1}) - 1 \right] \right\}   \nonumber \\
&= {\fall{n_j}{2} \over n (n_j + n_k - 1)} - {n_i n_j \over n (n_i + n_k)} - 2 {n_i n_j n_k \over n \fall{n_j + n_k}{2}} \nonumber \\
			& \hspace{0.3cm} + 2 {n_i n_j n_k \over \fall{n_j + n_k + 1}{3}} (H_n - H_{n_i - 1}). \label{eq:rt_O3_Aexp}
\end{align}
Following a similar line of computation as above, we can find a closed-form expression for $\beta(\n, i, j, k)$ as follows:
\begin{align}
\beta(\n, i, j, k)  
&{}={}  \sum_{m = 3}^{n} {1 \over m^2 (m-1)} \sum_{\substack{\bfzero \prec \m \preceq \n: \\ m_i = 1, |\m|=m}} {{\n \choose \m} \over {n \choose m}} m_j m_k  \nonumber \\
&={}  \sum_{m = 3}^{n} {1 \over m^2 (m-1)} {{n - n_i \choose m - 1} \over {n \choose m}} n_i \sum_{\substack{\bfzero \prec \m' \preceq \n - n_i \e_i: \\ |\m'|=m-1}} {{\n - n_i \e_i \choose \m'} \over {n - n_i \choose m - 1}} m'_j m'_k \nonumber \\
&={}  \sum_{m = 3}^{n} {1 \over m^2 (m-1)} {{n_j + n_k \choose m - 1} \over {n \choose m}} n_i {n_j n_k \over \fall{nj+n_k}{2}} \fall{m-1}{2}	\nonumber \\
&={}  \sum_{m = 1}^{n} {n_i \over n} {n_j n_k \over \fall{nj+n_k}{2}} {{n_j + n_k \choose m} \over {n \choose m}} \left( 1 - {2 \over m+1} \right)  \nonumber \\
&={}  {n_i \over n} {n_j n_k \over \fall{nj+n_k}{2}} \left\{{n_j + n_k \over n_i} - 2 \left[{n \over n_j+n_k+1} (H_n - H_{n_i - 1}) - 1\right] \right\}    \nonumber \\
&={}  {n_j n_k \over n (n_j + n_k - 1)} + 2 {n_i n_j n_k \over n \fall{n_j + n_k}{2}} - 2 {n_i n_j n_k \over \fall{n_j + n_k + 1}{3}} (H_n - H_{n_i - 1}), \label{eq:rt_O3_Bexp}
\end{align}
where the second equality above is the same change of variables from $\m$ to $\m' = \m - \e_i$ as in the $\alpha(\n, i, j, k)$ term. The third equality follows from identity \eqref{id:second_moment_mixed} in \factref{fact:restricted_moment}, and the second to last equality follows from Facts~\ref{fact:ratio} and \ref{fact:ratio2}.
Substituting \eqref{eq:rt_O3_Aexp} and \eqref{eq:rt_O3_Bexp} into \eqref{eq:rt_O3_breakup}, and using \eqref{eq:qt_rt} gives
\begin{eqnarray}
\qt(\n) &=& \Lambda(\n) \sum_{i, j, k \text{ distinct}} \left[ \pi_j P_{ji} P_{jk} \alpha(\n, i, j, k) + \pi_k P_{kj} P_{ji} \beta(\n, i, j, k)\right] \nonumber \\
&=& \Lambda(\n)  \sum_{i, j, k \text{ distinct}} \left\{\pi_j P_{ji} P_{jk} \left[ {\fall{n_j}{2} \over n (n_j + n_k - 1)} - {n_i n_j \over n (n_i + n_k)} - 2 {n_i n_j n_k \over n \fall{n_j + n_k}{2}} \right. \right. \nonumber \\
		&& \hspace{4.2cm} \left. + 2 {n_i n_j n_k \over \fall{n_j + n_k + 1}{3}} (H_n - H_{n_i - 1}) \right] \nonumber \\
        && \phantom{\Lambda(\n) \sum_{i, j, k \text{ distinct}}} +  \pi_k P_{kj} P_{ji} \left[ {n_j n_k \over n (n_j + n_k - 1)} + 2 {n_i n_j n_k \over n \fall{n_j + n_k}{2}} \right.	\nonumber \\
		&& \hspace{4.2cm}  \vphantom{\sum_{i, j, k \text{ distinct}}} \left. - 2 {n_i n_j n_k \over \fall{n_j + n_k + 1}{3}} (H_n - H_{n_i - 1}) \right] \Bigg\}.	\label{eq:rt_O3_qt}
\end{eqnarray}
Note that if $\bfP$ is reversible when restricted to the observed alleles $\Oset$, then \eqref{eq:rt_O3_qt} simplifies to the expression
given in \corref{cor:O3reversible}.
\hfill{}\qed

\subsection{Proof of \thmref{thm:O4}  ($\O=4$) \label{sec:combinatorial_proof_O4}}
Using \corref{cor:O3reversible}, we first note the following alternate expression for $\rt(\n)$ when $\O = 3$ and $\bfP$ is reversible restricted to the observed alleles:
\begin{equation} \label{id:cn3_symmetrized}
\rt(\n) = \sum_{i, j, k \text{ distinct}} {n_i \over n} \pi_i {P_{ij} P_{ik} \over 2}.
\end{equation}
Suppose $\O=4$ and assume that $\bfP$ is reversible restricted to the observed alleles $\Oset$.
Then using \propref{prop:rt_recursion}, we obtain
\begin{eqnarray}
\rt(\n) &=& \sum_{l, h \neq l} P_{hl} \sum_{\substack{\bfzero \prec \m \preceq \n: \\ m_l = 1}} {{\n \choose \m} \over {n \choose m}} {m_h \rt(\m - \e_l + \e_h) \over m (m-1)}    \nonumber \\
        &=& \sum_{l, h \neq l} P_{hl} \sum_{\substack{\bfzero \prec \m \preceq \n: \\ m_l = 1}} {{\n \choose \m} \over {n \choose m}} {m_h\over m (m-1)} 
\sum_{\substack{i, j, k \text{ distinct}, \\ i, j, k \neq l}} {m_i + \delta_{i,h} \over m} \pi_i {P_{ij} P_{ik} \over 2}    \nonumber \\
        &=& \sum_{i, j, k, l \text{ distinct}} {1 \over 2} \pi_i P_{ij} P_{ik} P_{il} \sum_{\substack{\bfzero \prec \m \preceq \n \\ m_l = 1}} {{\n \choose \m} \over {n \choose m}} {m_i (m_i + 1) \over m^2 (m-1)} \nonumber \\
        &&  + \sum_{i, j, k, l \text{ distinct}} \pi_i P_{ij} P_{ik} P_{jl} \sum_{\substack{\bfzero \prec \m \preceq \n \\ m_l = 1}} {{\n \choose \m} \over {n \choose m}} {m_i m_j \over m^2 (m-1)}, \label{eq:rt_O4_1}
\end{eqnarray}
where the second equality follows from using \eqref{id:cn3_symmetrized} since $\bfP$ is reversible when restricted to the alleles $\{i, j, k\}\subset\Oset$. Similar to the proof in \sref{sec:combinatorial_proof_O3}, if we define quantities $\zeta(\n, i, j, k, l)$ and $\delta(\n, i, j, k, l)$ as
\[
\zeta(\n, i, j, k, l) = \sum_{m = 4}^{n} {1 \over m^2 (m-1)} \sum_{\substack{\bfzero \prec \m \preceq \n: \\ m_l = 1, |\m|=m}} {{\n \choose \m} \over {n \choose m}} m_i (m_i + 1),
\]
and 
\[
\delta(\n, i, j, k, l) = \sum_{m = 4}^{n} {1 \over m^2 (m-1)} \sum_{\substack{\bfzero \prec \m \preceq \n: \\ m_l = 1, |\m|=m}} {{\n \choose \m} \over {n \choose m}} m_i m_j,
\]
then, using \eqref{eq:qt_rt} and \eqref{eq:rt_O4_1}, we obtain the following expression for $\qt(\n)=\Lambda(\n)\rt(\n)$:
\begin{align}
\qt(\n) = \Lambda(\n) \sum_{i, j, k, l \text{ distinct}}  \left[ \pi_i P_{ij} P_{ik} P_{il} {\zeta(\n, i, j, k, l) \over 2}
     + \pi_i P_{ij} P_{ik} P_{jl} \delta(\n, i, j, k, l) \right]. \label{eq:rt_O4_2}
\end{align}
By a very similar calculation to that in \sref{sec:combinatorial_proof_O3}, using Facts~\ref{fact:ratio} and \ref{fact:ratio2}, and  identities \eqref{id:second_moment_tweak} and \eqref{id:second_moment_mixed} in \factref{fact:restricted_moment} of the Appendix, we obtain the following closed-form expressions for $\zeta(\n, i, j, k, l)$ and $\delta(\n, i, j, k, l)$:
\begin{align*}
\lefteqn{\zeta(\n, i, j, k, l)}\hspace{10mm}\\
={}&    {n_l \over n} \left\{ {n_i + n_j + n_k \over n_l} {\fall{n_i}{2} \over \fall{n_i + n_j + n_k}{2}}  + {n_i \over n_j + n_k + n_l} \right. \\
& \hspace{1mm}  + {2 n_i (n_j + n_k) \over \fall{n_i + n_j + n_k}{2}} \left({n \over n_i + n_j + n_k + 1} (H_n - H_{n_l - 1}) -1 \right)     \\
& \hspace{1mm}  - \left[ {n_i + n_j \over n_k + n_l} {\fall{n_i}{2} \over \fall{n_i + n_j}{2}} + {2 n_i n_j \over \fall{n_i + n_j}{2}} \left({n \over n_i + n_j + 1} (H_n - H_{n_k + n_l - 1}) - 1 \right)  \right] \\
& \hspace{1mm} - \left.\left[ {n_i + n_k \over n_j + n_l} {\fall{n_i}{2} \over \fall{n_i + n_k}{2}} + {2 n_i n_k \over \fall{n_i + n_k}{2}} \left({n \over n_i + n_k + 1} (H_n - H_{n_j + n_l - 1}) - 1 \right)  \right]  \right\}.
\end{align*}
and
\begin{align*}
\lefteqn{\delta(\n, i, j, k, l)} \hspace{10mm} \\
={}&
    {n_l \over n} \left\{ {n_i + n_j + n_k \over n_l} {n_i n_j \over \fall{n_i + n_j + n_k}{2}} \right. \\
  &  \hspace{1mm} - {2 n_i n_j \over \fall{n_i + n_j + n_k}{2}} \left({n \over n_i + n_j + n_k + 1} (H_n - H_{n_l - 1}) -1 \right) \\
{}& \hspace{1mm}- \left.\left[ {n_i + n_j \over n_k + n_l} {n_i n_j \over \fall{n_i + n_j}{2}} - {2 n_i n_j \over \fall{n_i + n_j}{2}} \left({n \over n_i + n_j + 1} (H_n - H_{n_k + n_l - 1}) - 1 \right)  \right] \right\}.
\end{align*}
Simplifying the expression for $\delta(\n, i, j, k, l)$, we get the expression stated in \thmref{thm:O4}.
Observing that $\zeta(\n, i, j, k, l)$ is symmetric in $j$ and $k$, we see that for all $i, j, k$, and $l$ distinct in $\Oset$, 
\[
{\zeta(\n, i, j, k, l) + \zeta(\n, i, k, j, l) \over 2} = \gamma(\n, i, j, k, l) + \gamma(\n, i, k, j, l),
\]
where $\gamma(\n, i, j, k, l)$ is given by:
\begin{align*}
\gamma(\n, i, j, k, l) {}={}& {n_i \over n} \Bigg\{ \bigg[ {n_i - 1 \over 2 (n_i + n_j + n_k - 1)} - {2 n_j n_l \over \fall{n_i + n_j + n_k}{2}} \bigg] + {n_l \over 2(n_j + n_k + n_l)} \\
	&    \phantom{{n_i \over n} \left[ \right.}- \bigg[ {n_l (n_i - 1) \over (n_k + n_l)(n_i + n_j - 1)} - {2 n_j n_l \over \fall{n_i + n_j}{2}} \bigg] \Bigg\} \\
	& \hspace{-5mm} + {2 n_i n_j n_l \over \fall{n_i + n_j + n_k + 1}{3}} (H_n - H_{n_l - 1}) - {2 n_i n_j n_l \over \fall{n_i + n_j + 1}{3}} (H_n - H_{n_k + n_l - 1}).
\end{align*}
Using the fact that $\pi_i P_{ij} P_{ik} P_{il}$ is also symmetric in $j$ and $k$, we can then rewrite \eqref{eq:rt_O4_2} as
\[
\qt(\n) = \Lambda(\n) \sum_{i, j, k, l \text{ distinct}}  \left[ \pi_i P_{ij} P_{ik} P_{il} \gamma(\n, i, j, k, l)
     + \pi_i P_{ij} P_{ik} P_{jl} \delta(\n, i, j, k, l) \right].
\tag*{\hfill{}\qed}
\]

\section{Empirical study of accuracy \label{sec:numerical_accuracy}}
Here, we investigate the accuracy of approximating the sampling probability $q(\n)$ by using only the leading order term $\theta^{\O-1}\qt(\n)$.
In this study, we solve the recursion \eqref{eq:main_rec} numerically to obtain the true sampling probability $q(\n)$ for moderate sample sizes.

For a given sample $\n$, define the approximate sampling probability, $\qapprox(\n)$, by
\[
\qapprox(\n) = \theta^{\O - 1}\qt(\n).
\]
We can then define the relative error, $\err(\n)$, of the approximation $\qapprox(\n)$ from the true sampling probability $q(\n)$ as
\[
\err(\n) = {\lvert q(\n) - \qapprox(\n) \rvert \over q(\n)}.
\]
For a given sample size $n$, another natural measure of the approximation quality is the expected relative error under the distribution arising from the coalescent on samples of size $n$. Since $q(\n)$ is the probability of a particular ordered sample consistent with $\n$, the probability $p(\n)$ of the unordered sample $\n$, when sampling order is ignored, is given by 
\[
p(\n) = {n \choose n_1, \ldots, n_K} q(\n).
\]
We can then define the expected relative error for a sample size $n$ by $\exerr(n)$, given by
\begin{align*}
\exerr(n)   &= \sum_{\n: |\n| = n} p(\n) \err(\n)  = \sum_{\n: |\n| = n} {n \choose n_1, \ldots, n_K} \lvert q(\n) - \qapprox(\n) \rvert.
\end{align*}
We also define the worst-case relative error, $\worsterr(n)$, for a given sample size $n$ as the worse relative error among all samples of size $n$. Specifically,
\begin{align*}
\worsterr(n) &= \max_{\n: |\n| = n} \err(\n)  = \max_{\n: |\n| = n} {\lvert q(\n) - \qapprox(\n) \rvert \over q(\n)}.
\end{align*}

To study the accuracy of approximating $q(\n)$ by $\qapprox(\n)$, we examine the behavior of $\exerr(n)$ and $\worsterr(n)$ for a transition matrix estimated from real biological data.  Specifically, we use the reversible phylogenetic mutation rate matrix estimated in \cite[Table 1, matrix (1)]{yang:1994}  for the $\psi\eta$-globin pseudogenes of six primate species. Since their estimated matrix is a matrix of nucleotide substitution rates used for phylogenetic analysis, we rescale it by the minimum amount that can make it a valid Markov transition matrix. This rescaled matrix, denoted by $\bfPhat$, is given below to three digits of precision, and is used in our numerical experiments with different values of the mutation parameter $\theta$:
\begin{equation}
\bfPhat =
\left(
\begin{tabular}{llll}
0.433    & 0.398  & 0.074  & 0.095 \\
0.665    & 0.000  & 0.164  & 0.171  \\
0.074    & 0.098  & 0.394  & 0.434  \\ 
0.147    & 0.159  & 0.674  & 0.020
\end{tabular}
\right), \label{zhang_1a}
\end{equation}
in the $(T,C,A,G)$ basis.
The stationary distribution corresponding to this transition matrix is $\bfpihat=(0.308, 0.185, 0.308, 0.199)$ to three digits of precision.

\begin{figure}
\subfigure[]{
\includegraphics[trim=4mm 0mm 0mm 3mm,width=0.49\textwidth]{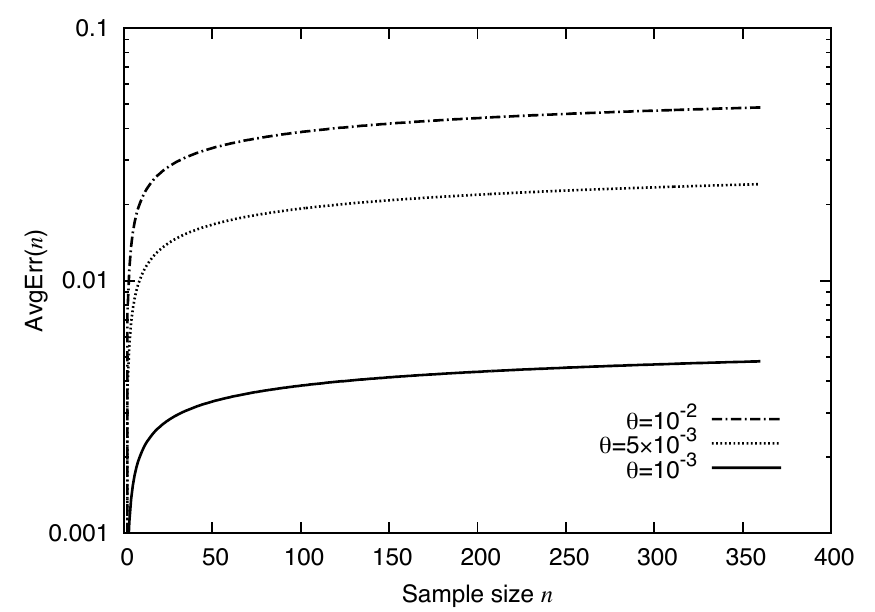}
\label{fig:avgerr}
}
\subfigure[]{
\includegraphics[trim=4mm 0mm 0mm 3mm,width=0.49\textwidth]{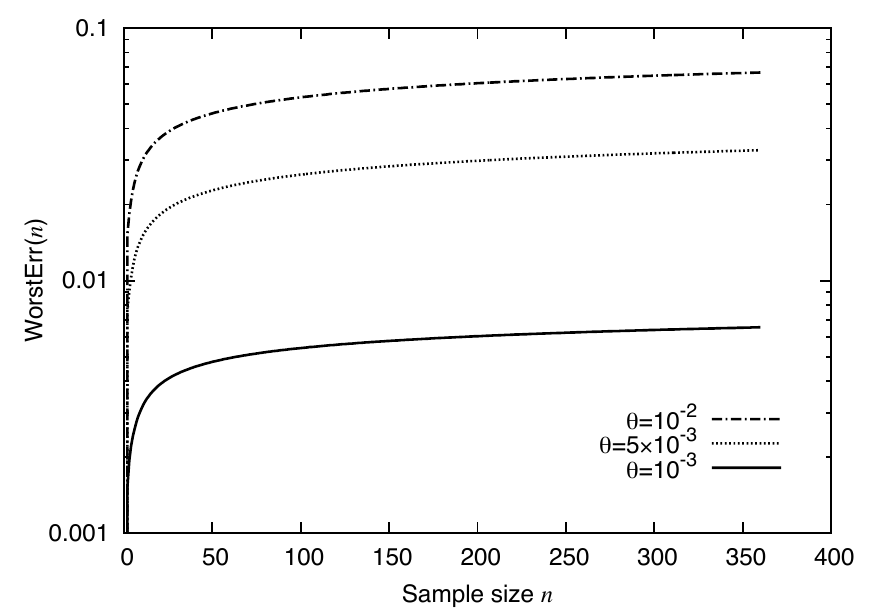}
\label{fig:worsterr}
}
\caption{Error plots as a function of the sample size $n$, for the transition matrix $\bfPhat$ in \eqref{zhang_1a} and mutation rate $\theta \in \{ 10^{-3}, 5 \times 10^{-3}, 10^{-2} \}$.
\subref{fig:avgerr}
The expected relative error, $\exerr(n)$.
\subref{fig:worsterr}
The worst-case relative error, $\worsterr(n)$.
}
\end{figure}

For many neutral regions of the human genome, typical mutation rates per base are in the range $10^{-3} \leq \theta \leq 10^{-2}$ \cite{nac:cro:2000}, and we consider $\theta\in\{10^{-3},5\times 10^{-3},10^{-2}\}$ in our study.  For the transition matrix in \eqref{zhang_1a}, the expected relative error $\exerr(n)$ and the worst-case relative error $\worsterr(n)$ are plotted in Figures~\ref{fig:avgerr} and \ref{fig:worsterr}, respectively, as  functions of the sample size $n$.
As can be seen from the plots, both the expected relative error and the worst-case relative error grow very slowly with the sample size $n$. 
Further, the ratio of $\worsterr(n)$ to $\exerr(n)$ is a small number between $1.3$ and $2.1$ for all $n \leq 360$, and is decreasing in $n$. Hence, it appears that the approximation quality of $\qapprox(\n)$ is uniformly good over all samples $\n$ for any given size $n$.

\section*{Acknowledgments}
We thank Paul Jenkins for useful discussion.  This research is supported in part by an NIH grant R01-GM094402, an Alfred P. Sloan Research Fellowship, and a Packard Fellowship for Science and Engineering.

\section*{Appendix}
Here, we provide some general combinatorial identities which are used several times for proving the main results in this paper.


\begin{fact} \label{fact:ratio}
For any positive integers $x, y, a$ and $b$ where $b \leq a$ and $x \leq y$,
\begin{equation}
\sum_{m=x}^{y} {{b \choose m} \over {a \choose m}} = {{a+1-x \choose a+1-b} - {a-y \choose a+1-b} \over {a \choose b}}. \label{id:ratio}
\end{equation}
\end{fact}

\begin{proof}
Starting from the left hand side of \eqref{id:ratio}, we have:
\begin{eqnarray*}
\sum_{m=x}^{y} {{b \choose m} \over {a \choose m}} 
    &=& {b! (a-b)! \over a!} \sum_{m=x}^{y} {a-m \choose a-b} \\
    &=& {{a+1-x \choose a+1-b} - {a-y \choose a+1-b} \over {a \choose b}},
\end{eqnarray*}
where the last equality follows from the standard combinatorial identity that for all positive integers $a, n$, and $k$, 
\[
\sum_{i=a}^{n}{n-i \choose k} = {n-a+1 \choose k+1}.
\tag*{\hfill{}\qed}
\]
\end{proof}



\begin{fact}    \label{fact:comb_harmonic}
For positive integers $a$ and $b$,
\begin{equation*}
  \sum_{m=1}^a {1 \over m} {a - m \choose b} = {a \choose b} (H_a - H_b).
\end{equation*}
\end{fact}

\factref{fact:comb_harmonic} can be verified by induction \cite{fu:1995} or by the method of Wilf-Zeilberger pairs \cite{gri:2003}.


\begin{fact}\label{fact:ratio2}
For positive integers $a$ and $b$ where $b \leq a$,
\begin{equation}
\sum_{m=1}^{b} {{b \choose m} \over {a \choose m}} {1 \over m+1} = {a+1 \over b+1} (H_{a+1} - H_{a-b}) - 1. \label{id:ratio2}
\end{equation}
\end{fact}

\begin{proof}
Starting from the left hand side of \eqref{id:ratio2}, we have:
\begin{eqnarray*}
\sum_{m=1}^{b} {{b \choose m} \over {a \choose m}} {1 \over m+1} 
    &=& {b! (a-b)! \over a!} \sum_{m=1}^{b} {a-m \choose a-b} {1 \over m+1}  \\
    &=& {1 \over {a \choose b}} \sum_{m=2}^{b+1} {a+1-m \choose a-b} {1 \over m}  \\
    &=& {1 \over {a \choose b}} \left[ \sum_{m=1}^{b+1} {a+1-m \choose a-b} {1 \over m}  -  {a \choose b} \right] \\
    &=& {1 \over {a \choose b}} \left[ {a+1 \choose b+1} (H_{a+1} - H_{a-b}) - {a \choose b} \right]  \\
    &=& {a+1 \over b+1} (H_{a+1} - H_{a-b}) - 1,
\end{eqnarray*}
where the fourth equality follows from using \factref{fact:comb_harmonic}.
\hfill{} \qed
\end{proof}


We also list some facts about the moments of a hypergeometric distribution which are appealed to several times in the paper.


\begin{fact}    \label{fact:general_moment}
If a multivariate hypergeometric distribution is parameterized by $\n = (n_1, n_2, \ldots, n_L)$, where $n = |\n|$, 
and a sample of size $m$, $\m = (m_1, m_2, \ldots, m_L)$, is drawn from it, then for any  $\t = (t_1, t_2, \ldots, t_L)$ where $t_i \geq 0$ for all $i$, $t = |\t|$ and $t \leq n$,
\begin{equation}\label{id:hyp_moment}
\mathbb{E}\left[ \prod_{i=1}^{L} \fall{m_i}{t_i} \right] = \sum_{\substack{\bfzero \preceq \m \preceq \n: \\ |\m|=m}}{{\n \choose \m} \over {n \choose m}} \prod_{i=1}^{L} \fall{m_i}{t_i} 
= {\prod_{i=1}^{L} \fall{n_i}{t_i} \over \fall{n}{t} } \fall{m}{t}
\end{equation}
\end{fact}

\begin{proof}
Starting from the middle term in \eqref{id:hyp_moment}, we get:
\begin{eqnarray*}
\sum_{\substack{\bfzero \preceq \m \preceq \n: \\ |\m|=m}}{{\n \choose \m} \over {n \choose m}} \prod_{i=1}^{L} \fall{m_i}{t_i}
&=&     \sum_{\substack{\bfzero \preceq \m \preceq \n: \\ |\m|=m}} {\prod_{i=1}^{L} \fall{n_i}{t_i} \over \fall{n}{t}} \fall{m}{t} {{\n - \t \choose \m - \t} \over {n - t \choose m - t}}  \\
&=&     {\prod_{i=1}^{L} \fall{n_i}{t_i} \over \fall{n}{t}} \fall{m}{t}  \sum_{\substack{\bfzero \preceq \m \preceq \n - \t: \\ |\m|=m - t}} {{\n - \t \choose \m} \over {n - t \choose m}}  \\
&=&     {\prod_{i=1}^{L} \fall{n_i}{t_i} \over \fall{n}{t}} \fall{m}{t},
\end{eqnarray*}
where the last equality follows because the term being summed is the probability mass function of a multivariate hypergeometric distribution parameterized by $\n-\t$, and the summation is over the entire domain of the distribution, and hence is 1.  \hfill{}\qed
\end{proof}


In the following fact, we compute some second moments of the hypergeometric distribution parameterized by $\n$ when restricted to those samples $\m$ which are non-zero at all types.


\begin{fact}    \label{fact:restricted_moment}
If $\n = (n_1, n_2, \ldots, n_L)$, where $n = |\n|$, and $1 \leq j \neq k \leq L$, then we have the following identities:
\begin{align}
    \sum_{\substack{\bfzero \prec \m \preceq \n: \\ |\m|=m}} {{\n \choose \m} \over {n \choose m}} m_j (m_j + 1) &= \sum_{\substack{T \subseteq [L]: \\ j \notin T}} (-1)^{|T|} \left[ {\fall{n_j}{2} \fall{m}{2} \over \fall{n-n_T}{2}} + {2 n_j m \over n - n_T} \right] {{n - n_T \choose m} \over {n \choose m}}     \label{id:second_moment_tweak} \\
    \sum_{\substack{\bfzero \prec \m \preceq \n: \\ |\m|=m}} {{\n \choose \m} \over {n \choose m}} m_j m_k &= \sum_{\substack{T \subseteq [L]: \\ j \notin T}} (-1)^{|T|} {m_j m_k \fall{m}{2} \over \fall{n-n_T}{2}} {{n - n_T \choose m} \over {n \choose m}} \label{id:second_moment_mixed}
\end{align}
\end{fact}

\begin{proof}
Applying the inclusion-exclusion principle and using \factref{fact:general_moment}, the identity in \eqref{id:second_moment_tweak} can be obtained as
\begin{align*}
\sum_{\substack{\bfzero \prec \m \preceq \n: \\ |\m|=m}} {{\n \choose \m} \over {n \choose m}} m_j (m_j + 1)
&= \sum_{\substack{T \subseteq [L]: \\ j \notin T}} (-1)^{|T|} \Bigg[ \sum_{\substack{\bfzero \preceq \m \preceq \n - \n_T: \\ |\m|=m}} {{\n - \n_T \choose \m} \over {n - n_T \choose m}} \left( \fall{m_j}{2} + 2 m_j \right) \Bigg] {{n - n_T \choose m} \over {n \choose m}} \\
&= \sum_{\substack{T \subseteq [L]: \\ j \notin T}} (-1)^{|T|} \left[ {\fall{n_j}{2} \fall{m}{2} \over \fall{n-n_T}{2}} + {2 n_j m \over n - n_T} \right] {{n - n_T \choose m} \over {n \choose m}}.
\end{align*}
Similarly for \eqref{id:second_moment_mixed}, we have
\begin{align*}
\sum_{\substack{\bfzero \prec \m \preceq \n: \\ |\m|=m}} {{\n \choose \m} \over {n \choose m}} m_j m_k
&= \sum_{\substack{T \subseteq [L]: \\ j, k \notin T}} (-1)^{|T|} \Bigg[ \sum_{\substack{\bfzero \preceq \m \preceq \n - \n_T: \\ |\m|=m}} {{\n - \n_T \choose \m} \over {n - n_T \choose m}} m_j m_k \Bigg] {{n - n_T \choose m} \over {n \choose m}}\\
&= \sum_{\substack{T \subseteq [L]: \\ j, k \notin T}} (-1)^{|T|} {m_j m_k \fall{m}{2} \over \fall{n-n_T}{2}} {{n - n_T \choose m} \over {n \choose m}}.
\tag*{\hfill{}\qed}
\end{align*}
\end{proof}


\bibliographystyle{apt}
\bibliography{14145_refs}

\end{document}